\documentclass[12pt,twoside,a4paper]{amsart}
\usepackage{amssymb}
\date{\today}

\def\I{{\mathcal I}}

\def\End{{\rm End}}
\def\ann{{\rm ann}}

\def\w{\wedge}

\def\dbar{\bar\partial}

\def\C{{\mathbb C}}
\def\w{{\wedge}}
\def\P{{\mathbb P}}

\def\supp{\text{supp}}
\def\D{{\mathcal D}}
\def\M{{\mathcal M}}

\def\S{{\mathcal S}}

\def\PM{{\mathcal{PM}}}

\def\H{{\mathcal H}}
\def\Hom{{\rm Hom\, }}
\def\Homs{{\mathcal Hom\, }}
\def\codim{{\rm codim\,}}

\def\Im{{\rm Im\, }}

\def\Ker{{\rm Ker\,  }}

\def\Exts{{\mathcal Ext}}
\def\Ext{{\rm Ext}}

\def\Ok{{\mathcal O}}

\def\L{{\mathcal L}}

\def\Re{{\rm Re\,  }}

\def\L{{\mathcal L}}
\def\U{{\mathcal U}}
\def\ann{{\rm ann\,}}

\def\Cu{{\mathcal C}}
\def\F{{\mathcal F}}

\def\CH{{\mathcal{CH}}}

\def\be{\begin{equation}}
\def\ee{\end{equation}}

\newtheorem{thm}{Theorem}[section]
\newtheorem{lma}[thm]{Lemma}
\newtheorem{cor}[thm]{Corollary}
\newtheorem{prop}[thm]{Proposition}

\theoremstyle{definition}

\theoremstyle{remark}

\newtheorem{preremark}{Remark}
\newtheorem{preex}{Example}

\newenvironment{remark}{\begin{preremark}}{\qed\end{preremark}}
\newenvironment{ex}{\begin{preex}}{\qed\end{preex}}

\numberwithin{equation}{section}

\title[]{Coleff-Herrera  currents,  duality,
and Noetherian operators}

\begin{document}

\date{\today}

\author{Mats Andersson}

\address{Department of Mathematics\\Chalmers University of Technology and the University of 
G\"oteborg\\S-412 96 G\"OTEBORG\\SWEDEN}

\email{matsa@math.chalmers.se}


\thanks{The author was
  partially supported by the Swedish Natural Science
  Research Council}

\begin{abstract}
Let $\I$ be a coherent subsheaf of a locally free sheaf $\Ok(E_0)$ and suppose that
$\F=\Ok(E_0)/\I$ has pure codimension. Starting with a residue current $R$ obtained from
a locally free resolution of $\F$  we construct
a vector-valued Coleff-Herrera current $\mu$ with support on the variety
associated to $\F$  such that $\phi$ is in $\I$ if and only if
$\mu\phi=0$. Such a current $\mu$ can also be derived  algebraically from a fundamental theorem of
Roos  about  the bidualizing functor, and the relation
between these two approaches is discussed. By a construction due to Bj\"ork
one gets Noetherian operators for $\I$ from the current $\mu$.
The current $R$ also provides an explicit realization of the Dickenstein-Sessa decomposition
and other related canonical isomorphisms.
\end{abstract}


\maketitle

\section{Introduction}

A function $\phi$ in the local ring $\Ok_0$ in one 
complex variable belongs to the ideal $I$ generated by $z^m$ 
if and only if 
$$
\L_\ell\phi(0)=0,\  \ell=0,\ldots, m-1,
$$
where  $\L_\ell=\partial^\ell/ \partial z^\ell$.
These conditions  can be elegantly expressed  by the single 
equation $\phi\dbar(1/z^m)=0$,
where $1/z^m$ is the usual principal value distribution. Moreover,
the current $\mu=\dbar(1/z^m)$ is canonical
up to a non-vanishing holomorphic factor. There is a well-known
multivariable generalization.
Let $f=(f^1,\ldots,f^p)$  be a tuple of 
holomorphic functions in a neighborhood  of the origin in $\C^n$ that 
defines  a complete   intersection, i.e., the codimension
of $Z^f=\{f=0\}$ is equal to $p$. Then the Coleff-Herrera product
$$
\mu^f=\dbar\frac{1}{f^1}\w\ldots\w\dbar\frac{1}{f^p},
$$
introduced in \cite{CH}, is a  $\dbar$-closed $(0,p)$-current  with support on  $Z^f$,
and it is independent (up to a nonvanishing holomorphic factor) of the
choice of generators of the ideal sheaf $\I$ generated by $f$. 
It was proved in \cite{DS} and
\cite{P1}  that $\I$ coincides with the ideal sheaf 
$\ann\mu^f$ of holomorphic functions $\phi$ such that  the current
$\mu^f\phi$ vanishes. This is often referred to as the duality principle.

\smallskip

The Coleff-Herrera product  is the model for a general
Coleff-Herrera current introduced by Bj\"ork:  Given a variety $Z$ of pure codimension
$p$ we say that a (possibly vector-valued) $(0,p)$-current $\mu$ (with support on $Z$)
is a Coleff-Herrera current on $Z$, $\mu\in \CH_Z$, if it is $\dbar$-closed,
annihilated by $\bar \I_Z$ (i.e., $\bar \xi\mu=0$ for each holomorphic 
$\xi$ that vanishes on $Z$),
and has the standard extension property SEP. This means, roughly speaking, that
 $\mu$ has no ``mass'' concentrated on any subvariety of higher codimension;
in particular that $\mu$ is determined by its values on $Z_{reg}$, see, e.g., 
\cite{JEB4} or \cite{A12}, and Section~\ref{chc}. The SEP  implies that
$\ann \mu$ has pure dimension, see, e.g., Proposition~5.3 in \cite{AW2}. 
The condition $\bar\I_Z\mu=0$ means that $\mu$ only involves holomorphic
derivatives.  Following Bj\"ork, see  \cite{JEB4}, 
one can quite easily   find  a finite number of
holomorphic differential operators $\L_\ell$ such that 
$\phi\mu=0$ if and only if
$\L_1\phi=\cdots =\L_\nu\phi=0$ on $Z$; i.e., 
a (complete) set of Noetherian  operators for 
$\ann\mu$. 
In this paper we use the residue theory developed in \cite{AW1} and \cite{AW2}
to extend the  duality for a complete intersection to a general
pure-dimensional ideal (or submodule of a locally free) sheaf.
In particular we can express  such an ideal 
as the annihilator of a  finite set of Coleff-Herrera currents
(Theorem~\ref{thm1} and its corollaries).
Jan-Erik Bj\"ork has pointed out 
to us that one can deduce  the same duality result from a fundamental theorem of
Jan-Erik Roos, \cite{Roos},  about purity for a module in terms
of the bidualizing sheaves,  combined with some  other known facts that 
will be described below. However our approach gives a representation
of the duality and the Coleff-Herrera currents in terms of
one basic residue current, that we first describe.

\smallskip

To begin with, let $\I$ be any  coherent subsheaf of a locally free sheaf $\Ok(E_0)$
over a complex manifold $X$, and assume that 
\begin{equation}\label{acomplex}
0\to \Ok(E_N)\stackrel{f_N}{\longrightarrow}\ldots
\stackrel{f_3}{\longrightarrow} \Ok(E_2)\stackrel{f_2}{\longrightarrow}
\Ok(E_1)\stackrel{f_1}{\longrightarrow}\Ok(E_0)
\end{equation}
is  a locally free resolution of $\F=\Ok(E_0)/\I$.
Here $\Ok(E_k)$ denotes the locally free sheaf
associated to the vector bundle $E_k$ over $X$.
If $X$ is Stein, 
then one can find such a resolution in a neighborhood of any given compact 
subset.
We will  assume that $\F$ has codimension $p>0$;   cf., Remark~\ref{pnoll}. 
Then $f_1$ is (can be assumed to be) generically surjective, 
and the analytic set $Z$ where it
is not surjective has codimension  $p$ and coincides with the 
zero set of the ideal sheaf $\ann\F$. 
In \cite{AW1} we defined, given Hermitian metrics
on $E_k$,  a residue current $R=R_p+R_{p+1}+\cdots $  with support on $Z$, 
where $R_k$ is a $(0,k)$-current that takes values in
$\Hom(E_0,E_k)$, such that a holomorphic section
$\phi\in\Ok(E_0)$ is in $\I$ if and only if  $R\phi=0$.  

Recall that $\F$ has {\it pure} codimension $p$ if
the associated prime ideals (of each stalk) all have codimension $p$.
The starting point in this paper is the following 
result that follows from  \cite{AW2}
(see also Section~\ref{k>p} below);
as  we will see later on
it is in a way  equivalent to Roos' characterization of purity.

\begin{thm}\label{oxy}
The sheaf $\F=\Ok(E_0)/\I$ has pure codimension $p$ if and only if 
$\I$ is equal to the annihilator of $R_p$, i.e.,
$$
\I=\{\phi\in\Ok(E_0);\  R_p\phi=0\}.
$$
\end{thm}

If $\F$ is Cohen-Macaulay we can choose a resolution
\eqref{acomplex} with  $N=p$, and then 
$R=R_p$ is a matrix of $\CH_Z$-currents 
which thus  solves our problem.   However, in general
$R_p$ is not $\dbar$-closed even if $\F$ has pure codimension.
Let
\begin{equation}\label{dualcomplex}
0\to \Ok(E_0^*)\stackrel{f_1^*}{\longrightarrow}\Ok(E_1^*)
\stackrel{f_2^*}{\longrightarrow} \ldots 
\stackrel{f^*_{p-1}}{\longrightarrow} \Ok(E_{p-1}^*)\stackrel{f^*_p}{\longrightarrow}
\Ok(E_p^*)\stackrel{f^*_{p+1}}{\longrightarrow}
\end{equation}
be the dual complex of \eqref{acomplex} and let 
\begin{equation}\label{extrep}
\H^k(\Ok(E^*_\bullet))=\frac{\Ker_{f^*_{k+1}}\Ok(E_k^*)}{f^*_k\Ok(E_{k-1}^*)}
\end{equation}
be the associated  cohomology sheaves.
It turns out that for each choice of $\xi\in\Ok(E_p^*)$ such that $f_{p+1}^*\xi=0$,
the current $\xi R_p$ is in $\CH_Z(E_0^*)$, and  we have in fact   a
bilinear (over $\Ok$)  pairing
\begin{equation}\label{pelle}
\H^p(\Ok(E_\bullet^*))\times\F\to \CH_Z,   \quad  (\xi,\phi)\mapsto \xi R_p\phi.
\end{equation}
Moreover, \eqref{pelle} is independent of the choice of Hermitian metrics
on $E_k$. 
It is well-known  that the sheaves in \eqref{extrep}   
represent the intrinsic  sheaves  $\Exts^k_\Ok(\F,\Ok)$.
(If $Z$ does not have pure codimension $p$  then we define
$\CH_Z$ as  $\CH_{Z'}$, where $Z'$ is the union of irreducible components of
codimension $p$; this  is reasonable, in view of the SEP.)

\begin{thm}\label{thm1} Assume that  $\F$ has  codimension $p>0$. 
The pairing  \eqref{pelle}  induces  an intrinsic  pairing
\begin{equation}\label{ext}
\Exts_\Ok^p(\F,\Ok)\times \F\to \CH_Z. 
\end{equation}
If $\F$ has pure codimension, then the pairing is non-degenerate.
\end{thm}

Notice that $\Homs(\F,\CH_Z)$   is the subsheaf of 
$\Homs(\Ok(E_0),\CH_Z)=\CH_Z(E_0^*)$ consisting of all Coleff-Herrera currents
$\mu$ with values in $E_0^*$ such that
$\mu\phi=0$ for all $\phi\in \I$. 
It follows that we have the equality
\begin{equation}\label{dual3}
\I=\{\phi\in\Ok(E_0);\   \mu \phi=0\ \text{for\ all}\ \mu\in\Homs(\F,\CH_Z)\}
\end{equation}
if $\F$ is pure. The 
sheaf   $\H^p(\Ok(E_\bullet^*))$ is coherent and thus locally finitely generated. 
Therefore we have now a solution to our problem:

\begin{cor}\label{labbe} Assume that $\F$ has pure codimension.
If $\xi_1,\ldots,\xi_\nu\in\Ok(E_p^*)$ generate $\H^p(\Ok(E^*_\bullet))$,
then  $\mu_j=\xi_j R_p$ are in $\Homs(\F,\CH_Z)$  and 
\begin{equation}\label{decomp}
\I=\cap_{j=1}^\nu\ann \mu_j.
\end{equation}
\end{cor}

\begin{remark}
If $\I$ is not pure,  one obtains a decomposition  \eqref{decomp}
after a preliminary decomposition $\I=\cap\I_\nu$,
where each $\I_\nu$ has pure codimension. 
\end{remark}

In case of a complete intersection, $\Exts^p(\F,\Ok)$ is isomorphic to $\F$ itself.
If  $\F=\Ok(E_0)/\I$ is   a sheaf of Cohen-Macaulay modules there is
also a certain symmetry:
If  \eqref{acomplex} is  a resolution with $N=p$, then
it is well-known, cf., also  Example~\ref{cmcase} below,   
that  the dual complex  \eqref{dualcomplex} is a resolution
of  $\Ok(E_p^*)/\I^*$, where  $\I^*=f_p^*\Ok(E_{p-1}^*)\subset\Ok(E_p^*)$,
and we have

\begin{cor}\label{chcor}
If $\Ok(E_0)/\I$ is Cohen-Macaulay,  then $\Ok(E_p^*)/\I^*$ is Cohen-Macaulay
as well and we have a non-degenerate  pairing
$$
\Ok(E_0)/\I\times \Ok(E_p^*)/\I^*\to \CH_Z, \quad  (\xi,\phi)\mapsto \xi R_p\phi.
$$
\end{cor}


\begin{remark}\label{pnoll}
Assume that $\F$ has codimension $p=0$, or equivalently, $\ann\F=0$.
If it is pure, i.e., $(0)$ is the only associated prime ideal, then
there is a homomorphism $f_0\colon\Ok(E_0)\to\Ok(E_{-1})$ 
such that $\I=\Ker f_0$. It is natural to  consider $f_0$ as a Coleff-Herrera
current $\mu$ associated with the zero-codimensional ``variety'' $X$.
Then 
 $\I=\ann \mu$ and thus analogues of 
Theorem~\ref{oxy}  and Corollary~\ref{labbe} still hold.  
\end{remark}

The duality discussed here  leads to a generalization of the 
Dickenstein-Sessa decomposition that we now  will describe.
It was proved by Malgrange,  see, e.g.,  \cite{JEB4}, 
that the analytic sheaf of distributions $\Cu$ is stalkwise injective.
Thus the double complex 
\begin{equation}\label{dubb}
\Homs_\Ok(\Ok(E_\ell),\Cu^{0,k})=\Cu^{0,k}(E_\ell^*),
\end{equation}
with differentials $\dbar$ and $f^*$,  is exact except at $k=0$ and  $\ell=0$,
where we have the cohomology sheaves $\Ok(E_\ell^*)$ and 
$\Homs(\F,\Cu^{0,\bullet})$, respectively.
By standard homological algebra, 
we therefore have natural isomorphisms
\begin{equation}\label{nathom}
\H^k(\Ok(E^*_\bullet),\Ok) 
\simeq
\H^k(\Homs(\F,\Cu^{0,\bullet})).
\end{equation}
The residue calculus also gives 

\begin{thm}\label{galopp}
Assume that $\codim\F=p>0$. Both mappings 
\begin{equation}\label{natmap}
\H^p(\Ok(E_{\bullet}^*))
\stackrel{\Psi}{\simeq}
\Homs(\F,\CH_Z)\simeq
\H^p(\Homs(\F,\Cu^{0,\bullet}))
\end{equation}
are isomorphisms, and the composed mapping coincides with  the  isomorphism \eqref{nathom}.
\end{thm}

These isomorphisms seem to be known as ``folklore'' since long ago,
cf., Section~\ref{bjork} below. Our  contribution should be the
proof by residue calculus, and  especially,  the
realization of the mapping $\Psi$ as $\xi\mapsto \xi R_p$.

\begin{ex} If $\mu\in\CH_Z$ is annihilated by $\I$ it follows that we have the
factorization $\mu=\xi R_p$. There are analogous isomorphisms
where $\Ok$ is replaced by $\Omega^r$, the sheaf of holomorphic $(r,0)$-forms,
and Coleff-Herrera currents of bidegree $(r,p)$, 
$\CH^r_Z=\CH_Z\otimes_{\Ok}\Omega^r$. 
For instance it follows that there is  a factorization
$$
[Z]=\xi R_p,
$$
where $[Z]$ is the Lelong current, and $\xi$ is in $\Omega^p(E^*_p)$
with  $f_{p+1}^*\xi=0$.
\end{ex}

\begin{ex}
We can rephrase the second isomorphism in \eqref{natmap} as  the decomposition
\begin{multline}\label{kraka}
{\Ker}\big(\Homs(\F,\Cu^{0,p})\stackrel{\dbar}{\to}(\Homs(\F,\Cu^{0,p+1})\big)=\\
=\Homs(\F,\CH_Z)\oplus \dbar\Homs(\F,\Cu^{0,p-1}).
\end{multline}
For a given $\dbar$-closed $(0,p)$-current $\mu$
(with values in $E_0^*$ and    annihilated by $\I$),
its canonical projection in  $\Homs(\F,\CH_Z)$ is given by $\xi R_p$,
where $\xi$ is obtained from $\mu$ via the isomorphism \eqref{nathom}.
\end{ex}

\begin{ex}
Assume that $Z$ has pure codimension $p$ and let
$\Cu_Z^{0,k}$ denote the sheaf of $(0,k)$-currents with support on $Z$.
If  $\F$ has support on $Z$, then  $\Homs(\F,\Cu^{0,k})=\Homs(\F,\Cu^{0,k}_Z)$.
Since any current with support on $Z$ must be annihilated
by some power of $\I_Z$, \eqref{kraka} implies  the 
  decomposition
\begin{equation}\label{ds}
\Ker \big(\Cu^{0,p}_Z\stackrel{\dbar}{\to}\Cu^{0,p+1}_Z\big)
=\CH_Z\oplus\dbar\Cu_Z^{0,p-1}
\end{equation}
that was first proved in \cite{DS} by Dickenstein and Sessa
(in the case of a complete intersection;   see \cite{JEB4}
for the general case).
\end{ex}

The main results are  proved in Section~\ref{proofs}.  
In Section~\ref{bjork} we sketch  a purely algebraic proof
of Theorem~\ref{thm1} (except for the explicit residue representation)
based on Roos' theorem.
In Section~\ref{absolut} we consider in some more detail the
absolute case, i.e., $p=n$, and in Section~\ref{coho}
we briefly discuss a cohomological variant of the duality.
In Section~\ref{k>p} we consider a partial generalization of \eqref{natmap}
to $k>p$; again we can trace residue manifestations of Roos' theorem.

In Section~\ref{recollect} we collect some basic material about
residue currents. 
For the reader's convenience we include 
Bj\"ork's  construction of Noetherian operators for the ideal $\ann \mu$.
To further exemplify the utility of the residue
calculus,  we include a  proof of Malgrange's theorem by means of
residues and integral formulas in Section~\ref{malgrange}.

\smallskip

All results above have natural analogues for polynomial ideals and modules:
Let $I$ be a submodule of $\C[z_1,\ldots,z_n]^r$, and assume that
$F=\C[z_1,\ldots,z_n]^r/I$ has positive codimension $p$.
From a free resolution of the of the corresponding homogeneous
module over the graded ring $\C[z_0,\ldots,z_n]$ we constructed
in \cite{AW1} a residue current on $\P^n$ whose restriction
$R$ to $\C^n$ has the property that $\Phi\in 
\C[z_1,\ldots,z_n]^r$  is in $I$ if and only if $R\Phi=0$ in $\C^n$.
If $F$ has pure codimension, then precisely as in the
semi-global case above we have that
$\Phi$ is in $I$ if and only if $R_p\Phi=0$. By the same proofs we get
complete analogues of Theorem~\ref{thm1} and its corollaries.
In particular if $F$ is pure,
we get  a finite number of global Coleff-Herrera
currents $\mu_j=\xi R_p$ such that $\Phi\in I$ if and only if
$\mu_j\Phi=0$ for each $j$.  Moreover, since $R_p$  has a current
extension to $\P^n$, following  the proof in Section~\ref{chc}
with $\Omega=\C^n$,  
we obtain for each $\mu_j$ a finite set of 
differential operators $\L_{j\ell}$ with polynomial
coefficients such that
$\L_{j\ell}\Phi$ vanishes on $Z$ for all $\ell$ if and only if
$\Phi$ is in the annihilator of $\mu_j$.
Starting with a primary decomposition of $I$ we obtain in this way
a complete proof of the existence of Noetherian 
operators for an arbitrary polynomial ideal, a fact first proved
by Ehrenpreis and Palamodov as the corner stone  in the
celebrated fundamental principle, see
\cite{Ehr}, \cite{Pal}, \cite{Hor}, and  \cite{JEB1}.
For a discussion about effectivity, see  \cite{Ob}.

\smallskip
{\bf Acknowledgement}  I am grateful to Jan-Erik Bj\"ork for 
invaluable discussions on these matters, and  for 
communicating  the arguments in Section~\ref{bjork}. I also
would like to thank the referee for important suggestions.

\section{Some elements of residue theory}\label{recollect}

Let $X$ be an $n$-dimensional complex manifold.
In \cite{AW2} was  introduced the sheaf of {\it pseudomeromorphic}  currents
$\PM$. Roughly speaking a current $\mu$ is pseudomeromorphic if locally
it is the push-forward under a modification $\pi\colon\tilde X\to X$
of a (finite sum of) currents like
$$
\dbar\frac{1}{s_1^{\alpha_1}}\w\cdots \w\dbar\frac{1}{s_q^{\alpha_q}}
\w \frac{\omega}{s_{q+1}^{\alpha_{q+1}}\cdots s_n^{\alpha_n}}
$$
where $s$ is a local coordinate system and  $\omega$ is a smooth form with compact support. 
Here $q$ may be $0$ which means that we have no residue factor but only
principal value factors. 
The sheaf $\PM$ is closed under $\dbar$ and multiplication with smooth forms.
It turns out that
if $\mu$ is in $\PM$ and $V$ is a subvariety, then the  restriction
of $\mu$ to the open set $X\setminus V$ has a natural extension to a pseudomeromorphic current
${\bf 1}_{X\setminus V}\mu$ on $X$ such that ${\bf 1}_V\mu:=\mu-{\bf 1}_{X\setminus V}\mu$
has support on $V$. If $h$ is any holomorphic tuple such that $V=\{h=0\}$, then
$\lambda\mapsto |h|^{2\lambda}\mu$, that is well-defined if $\Re\lambda>>0$, has
a current-valued  analytic continuation to $\Re\lambda>-\epsilon$, and the value at $\lambda=0$,
$|h|^{2\lambda}\mu|_{\lambda=0}$,  is precisely ${\bf 1}_{X\setminus V}\mu$.

If $\mu$ is in $\PM$ and has support on $V$, then $\bar I_V\mu=0$, i.e., $\bar\xi\mu=0$
for each holomorphic function that vanishes on $V$. If $\mu$ has support on $V$
we say that it has the standard extension property, SEP, if ${\bf 1}_W \mu=0$ for
each $W\subset V$ of positive codimension.  For (the equivalence to) the
 more classical way to  define SEP, see \cite{A12} Proposition~5.1. 
We also have the {\it dimension principle}: 
\begin{prop}\label{hyp}[\cite{AW2}, Corollary~2.4]
If $\mu\in\PM$ has  bidegree $(r,k)$ and  support on
a variety $V$ of codimension $>k$,  then $\mu=0$. 
\end{prop}

It follows that if $\mu$ has bidegree $(r,p)$ and support on $V$ with
codimension $p$ then it has automatically the SEP with respect to $V$.

\subsection{Coleff-Herra currents and Noetherian operators}\label{chc}

Let $V$ be a subvariety with pure codimension $p$. We define the sheaf of 
Coleff-Herra currents $\CH_V^r$ as the subsheaf of $\PM$ of currents of bidegree
$(r,p)$ that has support on $V$ and are $\dbar$-closed. See 
Proposition~5.1 in \cite{A12} for an  equivalent definition.

\begin{thm}[Bj\"ork \cite{JEB4}]
Let $V$ be a germ of an analytic variety of pure codimension $p$ at $0\in\C^n$.
There is a neighborhood $\Omega$ of $0$ such that for each
$\mu\in\CH_V(E_0^*)$ in $\Omega$, there are  holomorphic
differential operators $\L_1,\ldots,\L_\nu$ in $\Omega$ such that
for any $\phi\in\Ok(E_0)$,    
$\mu\phi=0$ if and only if
\begin{equation}\label{bjarne}
\L_1\phi=\cdots=\L_\nu\phi=0 \ \text{on}\ \ V.
\end{equation}
\end{thm}

\begin{proof}
In a suitable pseudoconvex neighhborhood $\Omega$ of $0$ 
we can find  holomorphic functions $f_1,\ldots,f_p$,  forming
a complete intersection,  such that
$V\cap \Omega$, henceforth denoted just $V$, 
is a union of irreducible components of $V_f=\{f=0\}$, and such that
$
df_1\w\ldots\w df_p\neq 0
$
on $V\setminus W$,  where $W$ is a hypersurface not containing any component
of $V_f$. 
By a suitable choice of coordinates $(\zeta,\omega)\in\C^{n-p}\times\C^p$ we may
assume that  $W$ is the zero set of 
$
h=\det (\partial f/\partial\omega).
$
Let
$
z=\zeta,  w=f(\zeta,\omega).
$
Since
$d(z,w)/d(\zeta,\omega)=h$,
locally in $\Omega\setminus W$,    $(z,w)$ is a holomorphic  coordinate system.
If we take the multiindex $M$ so large that $\mu$ is annihilated by
$f_j^{M_j+1}$, it  follows from
\cite{JEB4} (or Theorem~4.1 in \cite{A12}) that 
there is a  holomorphic function $A\in\Omega$ such that
$$
\mu=A \dbar\frac{1}{f_1^{M_1+1}}\w\ldots \w\dbar\frac{1}{f_p^{M_p+1}}
$$
in $\Omega$. Thus locally in $\Omega\setminus W$,
$$
\mu.\xi=\int_{w=0}\sum_{\alpha\le M}c_\alpha\frac{\partial^{M-\alpha}A(z,0)}{\partial w^{M-\alpha}}
\gamma\neg \frac{\partial^\alpha\xi}{\partial w^\alpha},
$$
where $\gamma\neg$ is contraction with the vector field
$$
\gamma=\frac{\partial}{\partial w_p}\w\ldots\w \frac{\partial}{\partial w_1}.
$$
Now  $\mu\phi=0$ if and only if for all test forms $\xi$,
\begin{equation}\label{lotta}
0=\mu\phi.\xi=
\int_{w=0}\sum_{\ell\le M} (Q_\ell\phi)\gamma\neg \frac{\partial^\ell}{\partial w^\ell}\xi,
\end{equation}
where
$$
Q_\ell=\sum_{\ell\le\alpha\le M}
c_{\alpha,\ell}\frac{\partial^{M-\alpha}A}{\partial w^{M-\alpha}}
\frac{\partial^{\alpha-\ell}}{\partial w^{\alpha-\ell}}.
$$
Applying to $\xi=w^\ell\eta$ 
(induction over $\ell$ downwards) it follows that \eqref{lotta}  holds for all 
$\xi$   if and only if
$Q_\ell\phi=0$ (locally) on $V\setminus W$ for all $\ell\le M$.

However, 
$\partial\omega/\partial w=(\partial f/\partial\omega)^{-1}=\gamma/h
$
where $\gamma$ is a holomorphic matrix in $\Omega$.
It follows that $\L_\ell=h^NQ_\ell$ are well-defined and holomorphic in $\Omega$
if $N$ is large enough, and by the SEP,   $\mu\phi=0$ in $\Omega$
 if and only if \eqref{bjarne} holds.
\end{proof}

One also get a global (in $\Omega$) representation of $\mu$ in this way:
Notice that for some $L$,  $\tilde\gamma=h^L\gamma$ is holomorphic
in $\Omega$. One can verify that (with $\L=\L_0$), actually
$$
\mu.\xi=\int_Z\frac{1}{h^{M+L}}\tilde\gamma\neg \L\xi
$$
for all test forms, if the right hand side is interpreted as a principal
value.

\subsection{Residue currents associated with Hermitian  complexes}\label{sec3}

 and pointed out that the currents  $U$ and $R$ are pseudomeromorphic. 
\bigskip

We first have to recall the construction in \cite{AW1}. Let
\begin{equation}\label{hcomplex}
0\to E_N\stackrel{f_N}{\longrightarrow}\ldots
\stackrel{f_3}{\longrightarrow} E_2\stackrel{f_2}{\longrightarrow}
E_1\stackrel{f_1}{\longrightarrow}E_0\to\cdots 
\stackrel{f_{-M+1}}{\longrightarrow} E_{-M}\to 0
\end{equation}
be a generically exact complex of Hermitian vector bundles over
a complex manifold $X$, and let
\begin{equation}\label{complex}
0\to \Ok(E_N)\stackrel{f_N}{\longrightarrow}\ldots
\stackrel{f_1}{\longrightarrow}\Ok(E_0)\longrightarrow 
\cdots \stackrel{f_{-M+1}}{\longrightarrow} \Ok(E_{-M})\to 0
\end{equation}
be the corresponding complex of locally free sheaves. Assume that
\eqref{hcomplex} is pointwise exact outside the variety $Z$. Furthermore, 
over $X\setminus Z$ let $\sigma_k\colon E_{k-1}\to E_k$ be the minimal
inverses of $f_k$. Then
$
f\sigma+\sigma f=I_E,
$
where $I_E$ is the identity on $E=\oplus E_k$, $f=\oplus f_k$ and $\sigma=\oplus\sigma_k$.
The bundle $E$ has a natural superbundle structure
$E=E^+\oplus E^-$, where $E^+=\oplus E_{2k}$ and $E^-=\oplus E_{2k+1}$,
and $f$ and $\sigma$ are odd mappings with respect to this structure,
see, e.g., \cite{AW1} for more details.
The operator   $\nabla=f-\dbar$ acts on $\Cu^{0,\bullet}(X, E)$ and extends to
a mapping $\nabla_\End$ on $\Cu^{0,\bullet}(X, \End E)$
and  $\nabla_{\End}^2=0$.
If  
$$
u=\frac{\sigma}{\nabla_{\End}\sigma}=\sigma+\sigma(\dbar\sigma)+\sigma(\dbar\sigma)^2+\cdots
$$
it turns out that $\nabla_\End u=I_E$ in $X\setminus Z$. 
One can define a canonical current extension $U$ of $u$
across $Z$ as the analytic continuation  to $\lambda=0$ of
$|F|^{2\lambda}u$, where $F$ is any holomorphic function that
vanishes on $Z$. In the same way we  can define
the current 
$R=\dbar|F|^{2\lambda}\w u|_{\lambda=0}$
with support on $Z$, and then 
\begin{equation}\label{res}
\nabla_\End U=I_E-R.
\end{equation}
More precisely
$$
R=\sum_\ell  R^\ell=\sum_{\ell k} R^\ell_k, 
$$
where $R^\ell_k$ is a $(0,k-\ell)$-current that takes values in $\Hom(E_\ell,E_k)$, i.e., 
$$
R^\ell_k\in\Cu^{0,k-\ell}(X,\Hom(E_\ell,E_k)).
$$
Moreover we have (Proposition~2.2 in \cite{AW1})
\begin{prop} \label{keps}
If $\phi\in\Ok(E_\ell)$ and
$f_\ell\phi=R^\ell\phi=0$, then $\phi= f_{\ell+1}\psi$ has local 
solutions $\psi\in\Ok(E_{\ell+1})$.
If $R^{\ell+1}=0$, then $\phi=f_{\ell+1}\psi$ has local holomorphic solutions
$\psi$  if and only if
$f_\ell\phi=R^\ell\phi=0$.
\end{prop}

Since \eqref{hcomplex} is generically exact, so is its dual complex
\begin{equation}\label{dualcomplex2}
0\to E_{-M}^*\stackrel{f^*_{-M+1}}{\longrightarrow}
\cdots \stackrel{f^*_N}{\longrightarrow}E^*_N\to 0
\end{equation}
of Hermitian vector bundles,
and we have the corresponding dual complex of locally free sheaves
\begin{equation}\label{dualcomplex3}
0\to \Ok(E_{-M}^*)\stackrel{f^*_{-M+1}}{\longrightarrow}
\cdots \stackrel{f^*_N}{\longrightarrow}\Ok(E^*_N)\to 0.
\end{equation}
Using the induced metrics, we get a residue current 
$$
R^*=\sum_k (R^*)^k=\sum_{k,\ell} (R^*)^k_{\ell},
$$
where $(R^*)^k_{\ell}$ takes values in $\Hom(E_k^*,E^*_\ell)$.
\begin{prop}\label{res2}
Using the natural isomorphisms $\Hom(E_k^*,E^*_\ell)=\Hom(E_\ell,E_k)$
we have that
$
(R^*)^k_\ell=R^\ell_k.
$
\end{prop}

\begin{proof} It is readily verified that the adjoint $\sigma^*\colon E^*\to E^*$ of
$\sigma\colon E\to E$ over $X\setminus Z$ is the minimal inverse of $f^*$. Therefore,
$$
u^*=(\sigma+\sigma(\dbar\sigma)+\sigma(\dbar\sigma)^2+\cdots)^*=
\sigma^*+\sigma^*(\dbar\sigma^*)+\sigma^*(\dbar\sigma^*)^2+\cdots,
$$
since, see \cite{AW1}, $\sigma^*\dbar\sigma^*=(\dbar\sigma^*)\sigma^*$. Now the proposition
follows.
\end{proof}

If $\xi\in\Ok(E_k^*)$ and $\phi\in\Ok(E_\ell)$ we write
$$
\xi R^\ell_k\phi=\phi(R^*)^k_\ell\xi.
$$

\subsection{The injectivity of the analytic sheaf   $\Cu$}\label{malgrange}

Here is a proof of Malgrange's theorem by residue calculus. 
Let $\F$ be any module over the local ring $\Ok_0$ and let
\eqref{acomplex} be a resolution of $\F$. We have to prove that then
the complex
\begin{equation}\label{mula}
0\to \Hom(\Ok_0(E_0),\Cu)\stackrel{f_1^*}{\longrightarrow}
\Hom(\Ok_0(E_1),\Cu) \stackrel{f_2^*}{\longrightarrow}
\end{equation}
is exact except at $k=0$. 
Fix a natural number $N$. Given a  smooth function $\phi$ in $X\subset\C^n$, let
$\tilde\phi$ be the function
$$
\tilde\phi(\zeta,\omega)=\sum_{|\alpha| <N}\partial^\alpha_{\bar\zeta}\phi(\zeta)(\omega-\bar\zeta)
^\alpha/\alpha!,
$$
in $\tilde X=\{(\zeta,\bar\zeta)\in\C^{2n};\ \zeta\in X\}$.
Then 
$$
\tilde\phi(\zeta,\bar\zeta)=\phi(\zeta),\quad  \dbar\tilde\phi=O(|\omega-\bar\zeta|^N).
$$
Moreover, if $f$ is holomorphic then 
$\widetilde{f\phi}=f\tilde\phi.$
Combining the formulas in \cite{A7} with the construction in \cite{A3},
we get 
$$
\tilde\phi(z,\bar z)=
\int_{\zeta,\omega} (f_{k+1}(z)H^kU^k+H^kR^k+H^{k}U^{k-1}f_k)
\w(\tilde\phi+\dbar\tilde\phi\w v^z)\w g,
$$
where $g$ is a suitable form in $\C^{2n}$ with compact support and
$v^z$ is the Bochner-Martinelli form in $\C^{2n}$ with pole at $(z,\bar z)$,
and $H^\ell$ are holomorphic forms.
Since $R^k=0$ for $k\ge 1$ when \eqref{acomplex} is a resolution, 
we have the  homotopy formula
$$
\phi = f_{k+1}T_{k+1}\phi+T_k(f_k\phi),\quad  k\ge 1,
$$
where
$$
T_k\phi(z)=\int_{\zeta,\omega}H^{k}U(\tilde\phi+\dbar\tilde\phi\w v^z)\w g^z.
$$
Moreover, as in \cite{A3} one can verify that
$T_k\phi$ is of class $C^M$ if $N$ is large enough.
If now $\mu$ has order at most $M$, then we have
$$
\mu=T^*_{k+1}f^*_{k+1}\mu+f_k^*T^*_k\mu,
$$
so if $f^*_{k+1}\mu=0$, then $\mu=f^*_k\gamma$  if $\gamma=T^*_k\mu$.
Thus \eqref{mula} is exact at $k$.

\section{Proofs of the main results}\label{proofs}

Assume that $\F$ is a coherent sheaf of positive codimension $p$, and let
\eqref{acomplex} 
be a (locally) free resolution of $\F=\Ok(E_0)/\I$. 
Moreover, assume that  $f_1$ is generically
surjective so that  the corresponding vector bundle complex 
\begin{equation}\label{nupcomplex}
0\to E_N\stackrel{f_N}{\longrightarrow}\ldots
\stackrel{f_3}{\longrightarrow} E_2 \stackrel{f_2}{\longrightarrow}
E_1 \stackrel{f_1}{\longrightarrow} E_0\to 0
\end{equation}
is generically exact. 
It follows from  Proposition~\ref{hyp}  that
$$
R^0=R^0_p+R^0_{p+1}+\cdots .
$$
By  Theorem~3.1 in \cite{AW1},  $R^\ell_k=0$ for each $\ell\ge 1$,
i.e., $R=R^0$,   and
combining with  Proposition~\ref{keps} above we find that
a $\phi\in\Ok(E_0)$ is in $\I$ if and only if $R\phi=0$.
It is proved in Section~5 of  \cite{AW2} that 
$\F$ has pure codimension $p$ if and only if 
$\ann R=\ann R_p$, i.e.,    Theorem~\ref{oxy} holds.

\begin{proof}[Proof of Theorem~\ref{thm1}]
It follows from \eqref{res} that 
\begin{equation}\label{opera}
\dbar R_k=f_{k+1}R_{k+1}
\end{equation}
for each $k$.
If $\xi\in\Ok(E_k^*)$ and $f^*_{k+1}\xi=0$  we therefore have
$$
\dbar(\xi R_k)=\pm \xi\dbar R_k=\pm \xi f_{k+1} R_{k+1}=\pm (f_{k+1}^*\xi)R_{k+1}=0.
$$
Thus $\xi R_p$ is $\dbar$-closed
and since it is also  pseudomeromorphic, cf., Proposition~\ref{hyp},   it 
is in $\CH_Z$. Moreover, if
$\xi=f_p^*\eta$, then 
$$
\xi R_p=(f^*_p\eta)R_p=\eta f_pR_p=\eta\dbar R_{p-1}=0
$$
since $R_k=0$ for $k<p$. Thus $\xi R_p$ only depends on the cohomology class
of $\xi $ in $H^p(\Ok(E_\bullet^*))$. 
We now choose another Hermitian metric on $E$ and let $\tilde R$ denote the current 
associated with  the new metric.
It is showed  in \cite{AW1} (see the proof of Theorem 4.4) that then
$$
R_p-\tilde R_p=f_{p+1}M_{p+1}^0
$$
for a certain residue current $M$. Thus $\xi R_p=\xi\tilde R_p$.
It follows that the mapping \eqref{pelle}  is
well-defined and independent of the Hermitian metric on $E$.

It is enough to prove the invariance at a fixed point $x$, so we consider
stalks of the  sheaves at $x$.  It is well-known that then  our resolution
$\Ok_x(E_\bullet), f_\bullet$ can be written   
$$
\Ok_x(E'_\bullet\oplus E''_\bullet)\simeq\Ok_x(E'_\bullet)
\oplus \Ok_x(E''_\bullet),\    f_\bullet=f'_\bullet\oplus f''_\bullet,
$$
where $\Ok_x(E'_\bullet)$ is a resolution of $\F_x$ and (since we assume
that $E_0$ has minimal rank)  $\Ok_x(E''_k), k\ge 1$,  is a resolution of 
$\Ok_x(E_0'')=0$. 
It follows that the natural mapping 
$H^p(\Ok_x((E'_\bullet)^*)\to H^p(\Ok_x((E_\bullet)^*)), \  \xi'\mapsto (\xi',0)$,  
is an isomorphism. Moreover,
if we choose a metric on $E_k=E'_k\oplus E''_k$ that respects the direct
sum, then the resulting current $R$ is $R'\oplus 0$, where $R'$ is the
current associated with $\Ok_x(E'_\bullet)$. Since all minimal
resolutions are isomorphic, the  mapping \eqref{ext} is therefore well-defined.

\smallskip

It remains to check that  \eqref{ext} is non-degenerate.
If $\xi\in\Ok(E_p^*)$ with $f^*_{p+1}\xi=0$ and $\xi R_p\phi=0$ for all
$\phi\in\Ok(E_0)$, then clearly $\xi R_p=0$.
Since $R=R^0_p$, by Proposition~\ref{res2} therefore
$(R^*)^p_\ell\xi=0$ for all $\ell$,  and now it follows from 
Proposition~\ref{keps} that $\xi=f_p^*\eta$ for some $\eta$.
Thus (the class of)  $\xi$ is zero in $\H^p(\Ok(E^*_\bullet))$.

Now, assume that  $\xi R_p\phi=0$ for all $\xi$ such that $f_{p+1}^*\xi=0$.
If  $\F$ is Cohen-Macaulay and  $N=p$,  then
$f^*_{p+1}=0$ so the assumption  implies that $R_p\phi=0$, and thus
$\phi\in\I$.  However,
generically on $Z$, $\F$ is Cohen-Macaulay, 
and hence for an arbitrary resolution
we must have that $R_p\phi=0$ outside a variety of codimension $\ge p+1$. 
Since $R_p\phi$ is pseudomeromorphic with bidegree $(0,p)$ 
it follows  from Proposition~\ref{hyp} that $R_p\phi$ vanishes identically. 
If we in addition assume that $\F$ has pure codimension it follows
from  Theorem~\ref{oxy} that  $\phi\in\I$. Thus the pairing is
non-degenerate. 
\end{proof}

\begin{ex}[The Cohen-Macaulay case]\label{cmcase}
It is well-known, see, e.g., \cite{Eis1}, 
that $\F$ is Cohen-Macaulay if and only if 
it admits resolutions of length $p=\codim Z$.
If \eqref{acomplex} is a resolution with $N=p$, then
$R=R^0_p$, and hence $R^*=(R^*)^0_p$. It follows from
Proposition~\ref{keps},  applied to $R^*$, that the dual complex
\eqref{dualcomplex} is a resolution of $\Ok(E_0^*)/\I^*$ and,  in particular,
that  $\Ok(E_0^*)/\I^*$ is Cohen-Macaulay.
\end{ex}

\begin{proof}[Proof of  Theorem~\ref{galopp}]
Let
$$
\L^\nu=\sum_{\ell+k=\nu} \Cu^{0,k}(E_\ell^*)
$$
be the  total complex with differential $\nabla^*=f^*-\dbar$,
associated with the double complex \eqref{dubb}.
We thus have   natural isomorphisms
\begin{equation}\label{isos}
\H^k(\Ok(E_\bullet^*))\simeq \H^k(\L)=_{def}
\frac{\Ker_{\nabla^*}\L^k}{\nabla^*\L^{k-1}}\simeq
\H^k(\Homs(\F,\Cu^{0,\bullet})).
\end{equation}
The naturality means that the ismorphisms  are induced by the natural mappings
$\Ok(E_k^*)\to\L^k$ and $\Homs(\F,\Cu^{0,\ell})\to \L^k$, respectively, and that
$\xi\in\Ok(E^*_k)$ such that $f^*_{k+1}\xi=0$
defines the same class as $\mu\in\Hom(\F,\Cu^{0,k})$ with $\dbar\mu=0$ if and
only if there is $W\in\L^{k-1}$ such that
$
\nabla^*W=\xi-\mu.
$

If now $\xi\in\Ok(E_k^*)$ and $f^*_{k+1}\xi=0$, then $\nabla^*\xi=0$, and hence
\begin{equation}\label{astrid}
\nabla^* (U^*)^k\xi=\xi- (R^*)^k\xi=\xi-\xi R_k,
\end{equation}
cf., \eqref{res} and Proposition~\ref{res2} above. 
Therefore  the composed mapping
in \eqref{natmap} coincides with the isomorphisms  in \eqref{nathom}.
It is readily verified that the second mapping in \eqref{natmap}  is injective,
see, e.g., Lemma~3.3  in \cite{A12}, and hence
both mappings  must be isomorphisms. Thus
Theorem~\ref{galopp} is  proved.
\end{proof}

We think it may be  enlightening with  a proof of the first isomorphism 
in \eqref{natmap}  that does not rely on Malgrange's theorem.
We already know from Theorem~\ref{thm1} that this mapping  is injective,
so we have to prove the surjectivity.
The proof 
is based on the following lemma.

\begin{lma}\label{chlma} 
If there is a current $W\in\L^{p-1}$ such that 
$\nabla^*W=\mu\in\CH_Z(E_0^*)$, then $\mu=0$. 
\end{lma}

\begin{proof}
Let $u$ be a  smooth form $u$ such that $\nabla^*_{\End}u=I_{E^*}$ in $X\setminus Z$.
For a given neighborhood $\omega$ of $Z$, take a cutoff function $\chi$
with support in $\omega$ and equal to $1$ in some neighborhood of $Z$.
Then $g=\chi I_{E^*}-\dbar\chi\w u$ is smooth with compact
support in $\omega$, equal to $I_{E^*}$ in
a neighborhood of $Z$, and moreover $\nabla^* g=0$.
Therefore, $\nabla^* (g W)=g \mu=\mu$ and hence, for degree reasons, 
we have a solution  $\dbar w=\mu$ with support in $\omega$.
Since $\omega\supset Z$ is arbitrary it follows, cf., Lemma~3.3 in \cite{A12},
that $\mu=0$.
\end{proof}

Since \eqref{dubb} is  exact in $k$ except at $k=0$,
the  first equivalence in \eqref{isos} holds.
Take  $\mu\in\Homs(\F,\CH_Z)$. Then  $\nabla^*\mu=(f_1^*-\dbar)\mu=0$
so by \eqref{isos} (with $k=p$)  there is $\xi\in\Ok(E_p^*)$ such that
$
\nabla^* W=\xi-\mu
$
has a current solution $W\in\L^{p-1}$. In view of 
\eqref{astrid} it now follows from  Lemma~\ref{chlma} that
$\mu=\xi R^0_p$.

\section{An algebraic approach}\label{bjork}

In this section we indicate  how Theorem~\ref{thm1} and its corollaries
(except for the concrete representation $\xi\mapsto \xi R_p$) 
can be proved  algebraically. This  material has 
been communicated to us by Jan-Erik Bj\"ork.
Whereas our residue proof above was based on   Theorem~\ref{oxy},
the algebraic proof instead relies on the following 
fundamental result due to  J-E Roos, \cite{Roos}:

\begin{thm}\label{roosthm}
The sheaf $\F$  has pure codimension $p$ if and only if
the natural mapping
\begin{equation}\label{roos}
\F\to\Exts^p(\Exts^p(\F,\Ok),\Ok)
\end{equation}
is injective.
\end{thm}

Assume that 
$$
0\to \I\to\Ok(E_0)\to\F \to 0
$$ 
is exact as before. Moreover, 
let us  assume to begin with that we already know the isomorphisms \eqref{natmap}.
In particular we then have that 
$$
\Exts^p(\F,\Ok)\simeq\Homs(\F,\CH_Z).
$$
Thus  we can choose (locally) a finite number of generators $\mu_\alpha$, $\alpha\in A$,  for 
$\Homs(\F,\CH_Z)$ and get an  exact sequence
$$
0\to \I\to \Ok^A\to \Homs(\F,\CH_Z)\to 0,
$$
and therefore we  have, with $\M=\Homs(\F,\CH_Z)\simeq\Exts^p(\F,\Ok)$,  that
$$
\Exts^p(\M,\Ok)\simeq\Homs(\M,\CH_Z).
$$
We claim that the canonical mapping \eqref{roos} is given by
$$
\Ok(E_0)\ni\phi\mapsto (\mu_\alpha\phi)_\alpha.
$$
In fact,  clearly it is a mapping from $\F=\Ok(E_0)/\I$, since each $\mu_\alpha$ is.
Moreover, if $(\psi_\alpha)\in\I$, then by definition
$\sum_\alpha\psi_\alpha\mu_\alpha=0$, and hence $(\mu_\alpha\phi)_\alpha$ defines an element
in 
$$
\Homs(\Ok^A/\I,\CH_Z)\simeq\Exts^p(\Exts^p(\F,\Ok),\Ok).
$$
One can  verify that this mapping is independent of the choice of generators,
and  must be the canonical mapping \eqref{roos}.

It follows that \eqref{roos} is injective if and only if 
the equality \eqref{dual3} holds. If $\F$ has pure 
codimension $p$, by Theorem~\ref{roosthm} therefore
\eqref{dual3} holds and then 
Corollary~\ref{labbe}  as well as Theorem~\ref{thm1} follow.


\begin{remark} 
We actually get a residue  proof of Theorem~\ref{roosthm}: 
If $\F$ has pure codimension $p$ we know that  \eqref{dual3} holds
by the residue theory, and thus  \eqref{roos} is injective.
On the other hand, it is not hard to see, e.g., it follows from \cite{AW2}, that 
the annihilator of (a set of) currents in $\CH_Z$ 
must have pure codimension $p$. Thus  the injectivity  of
\eqref{roos} implies that  $\F$ has pure codimension.
\end{remark}

Let us conclude with a brief discussion of  \eqref{natmap}.
The  Dickenstein-Sessa decomposition \eqref{ds} 
is well-known;  see \cite{DS} in  case of a complete
intersection and \cite{JEB4} for the general case.
Malgrange also proved that  $\Cu_Z$ is stalkwise injective as an analytic sheaf.
Using these two facts and considering  the spectral
sequence obtained from the double complex
$$
\Homs(\Ok(E_\ell),\Cu_Z^{0,k}),
$$
one can conclude that the second mapping in \eqref{natmap} is indeed
an isomorphism, and hence both of them. However, we omit the details.

\section{The absolute  case and Bezoutians}\label{absolut}

Let $I\subset\Ok_0(E_0)$ be a submodule of the free module $\Ok_0(E_0)$
over the local ring $\Ok_0$  such that the zero variety of $\ann(\Ok_0(E_0)/I)$
is  $Z =\{0\}$. 
Moreover, let \eqref{acomplex} be a resolution of $\Ok_0(E_0)/I$ of length $n$.
From Corollary~\ref{chcor} we have the non-degenerate pairing
$$
\Ok_0(E_0)/I\times\Ok_0(E_n^*)/I^*\to\CH_{\{0\}}.
$$
Let $\alpha\in\Omega^n_0$ be a germ of a
nonvanishing holomorphic $(n,0)$-form at the origin,
and let 
$$
\CH_{\{0\}}\to\C,\quad \mu\to \mu.\alpha=\int\alpha\w \mu.
$$
Then we have
\begin{prop} The composed mapping 
\begin{equation}
\Ok_0(E_0)/I\times\Ok_0(E_n^*)/I^*\to \C
\end{equation}
is a non-degenerate pairing.
\end{prop}

\begin{proof}
If $\phi\in\Ok_0(E_0)$  is not in $I$, then there is some 
$\xi\in\Ok_0(E_n^*)$ such that
$\mu=\xi R_n\phi$ is not identically zero. Since $\mu$ is in
$\CH_{\{0\}}$   there is some holomorphic
$\psi\in\Ok_0$ such that $\psi\mu\neq 0$.  Thus 
$
\psi\xi R_n \phi.\alpha=\psi\mu.\alpha\neq 0.
$
Since we can interchange the roles of $I$ and $I^*$ the proposition follows. 
\end{proof}

In particular, we obtain an (non-canonical) isomorphism
\begin{equation}\label{holm}
(\Ok_0(E_0)/I)^*\simeq \Ok_0(E_n^*)/I^*.
\end{equation}

Notice that the form $u_n^0$ defines a $\Hom(E_0,E_n)$-valued
Dolbeault cohomology class $\omega$ in $\U\setminus\{0\}$.  
Since  $\dbar U_n^0=R_n$ we  have
\begin{equation}\label{svarv}
(\xi,\phi)=\int_{|\zeta|=\epsilon}\alpha\w \xi \omega \phi
\end{equation}
From  (5.4) in \cite{AW1} we get the representation 
$$
\phi(z)=f_1(z)\int H^1U\phi\w g +\int H^0_nR_n\phi\w g, \quad \phi\in\Ok(E_0);
$$
here $H^1$ is a holomorphic $\Hom(E,E_1)$-valued form,
$H_n^0$ is a  $\Hom(E_n,E_0)$-valued holomorphic $(n,0)$-form,
so that  
$$
H_n^0=h_n^0\alpha,
$$
and $g$ is the  form (5.2) in \cite{AW1}.
It has compact support and 
depends  holomorphically on $z$. Moreover,
$g=\chi+\cdots,$
where the dots  denote smooth forms of positive bidegree, so
modulo $I$ we  have
$$
\phi(z)\equiv_{mod I}\int_{|\zeta|=\epsilon}H_n^0(\cdot,z)\omega\phi,
$$
and hence
$$
\phi(z)\equiv_{mod I} (h_n^0(\cdot, z),\phi).
$$
This  means that we  can consider $h_n^0$ as a (generalized)  Bezoutian, cf., \cite{CD}. 
For  each analytic functional $\mu$ on $\Ok_0(E_0)$ 
that vanishes on  $I$  there is a unique
element $\xi$ in $\Ok_0(E_n^*)/I^*$ such that the action on $\Ok_0(E_0)$ modulo
$I$ coincides with
$\mu.\phi=(\xi,\phi)$ in view of \eqref{holm}.
More explicitly we have 
$$
\xi(\zeta)=\mu_z(h_n^0(\zeta,\cdot)).
$$

In the classical case of a complete intersection
$I=(f^1,\ldots,f^n)$, if we choose Hefer forms,
i.e., $(1,0)$-forms $h_k=\sum_j h_{jk} d\zeta_j$ such that
$\sum h_{jk}(\zeta_j-z_j)=f^k(z)-f^k(\zeta)$, 
and $\alpha=d\zeta_1\w\ldots\w d\zeta_n$, then, cf., \cite{A7}, 
it turns out that $h_n^0=\det(h_{jk})$; this  is a well-known
formula, cf., \cite{CD}, for the   Bezoutian in this case.

\section{Cohomological  residues}\label{coho}

The Coleff-Herrera currents admit a nice intrinsic way
of expressing the action of holomorphic differential operators,
but the very definition relies on Hironaka's theorem
about the existence of resolutions of singularities.
In \cite{DS} and \cite{P1} there is also
a cohomological way of expressing the duality for a complete
intersection.
We have the following cohomological version  of
Corollary~\ref{labbe}.

\begin{thm}\label{gurka}
With the assumptions and notation as in  Corollary~\ref{labbe}
and with $w_j=\xi_j u^0_p$ outside $Z$ we have for $\phi\in\Ok(E_0)$
that
\begin{equation}\label{bulle}
w_j\phi.\psi=0,  \quad \psi\in\D,\quad \dbar\psi=0 {\text \  close\ to\ } Z
\end{equation}
if and only if $\phi\in \I$. 
\end{thm}

In fact, $w_j$ can be extended to $W_j=\xi_j U^0_p$ across
$Z$ and 
$$
\dbar W_j =\dbar (\xi_j U^0_p)=\pm \xi_j R_p=\pm \mu_j.
$$
One can now verify that \eqref{bulle} implies that $\mu_j\phi=0$,
see, e.g.,  Theorem~6.1 in \cite{A12}. 
Thus Theorem~\ref{gurka} follows from  Corollary~\ref{labbe}.
Moreover,  \eqref{bulle} is equivalent
to that $w_j\phi$ are $\dbar$-exact in $X\setminus Z$ for
each Stein neighborhood $X$.

\smallskip
In the complete intersection case, as well as the Cohen-Macaulay case,
see \cite{JL}, the proof is algebraic and only involves ``cohomological'' residues,
whereas the proof of Theorem~\ref{gurka} here is
obtained via the residue calculus. 
It is reasonable to believe that one can
produce  a purely algebraic proof (thus avoiding Hironaka's theorem) based on 
Roos' theorem,  cf., Section~\ref{bjork}.

\section{Higher $\Exts$ sheaves}\label{k>p}

Let $\F=\Ok(E_0)/\I$ has codimension $p$ as before.
 In view of 
Proposition~\ref{hyp} and \eqref{natmap} one could guess that
$\Exts^k(\F,\Ok)$ for $k>p$ could be represented by
cohomology of pseudomeromorphic currents. We have
the following partial result.

\begin{thm}\label{orm}
Assume that \eqref{acomplex} is a resolution of $\F$ and $R$ is the 
associated residue current (with respect to some given metric).
For each $k$, $\Ok(E_k^*)\ni\xi\mapsto \xi R_k$
induces an invariant injectice mapping
\begin{equation}\label{kmap1}
\Exts^k(\F,\Ok)\to \H^k(\Homs(\F,\PM^{0,\bullet})).
\end{equation}
Moreover,  the composed mapping
\begin{equation}\label{kmap2}
\H^k(\Ok(E_\bullet^*))\to 
\H^k(\Homs(\F,\PM^{0,\bullet}))\to
\H^k(\Homs(\F,\Cu^{0,\bullet}))
\end{equation}
coincides with the natural isomorphism \eqref{nathom}. 
\end{thm}

\begin{remark}\label{anka}
If we widen the definition of $\PM$ slightly so that it is
preserved under any surjective holomorphic mapping
rather than just modifications, then 
the $\dbar$-complex $\PM_{0,\bullet}$ is exact and thus it is 
a fine resolution of the sheaf $\Ok$. It is reasonable to believe
that  $\PM$ so defined is stalkwise injective. If this is true,
by considering the double complex $\PM_{0,k}(E^*_\ell)$, we could conclude
that the first mapping in \eqref{kmap2} is an isomorphism for any $k$,
and hence that both mappings are.
\end{remark}

Let $Z_k$ be the analytic set where the mapping $f_k$ in \eqref{acomplex}
does not have optimal rank. It is well-known and not hard to see that
these sets are independent of the resolution and hence invariants
of the sheaf $\F$. Then
$$
\cdots Z_{k+1}\subset Z_k\subset\cdots \subset Z_{p+1}\subset Z_{sing}\subset Z_p=Z
$$
and the Buchsbaum-Eisenbud theorem,
see \cite{Eis1}, states that the codimension of $Z_k$ is at least $k$.
Moreover, by Corollary~20.14 in \cite{Eis1}, $\F$ has pure codimension $p$
if and only if 
\begin{equation}\label{plast}
\codim Z_k\ge k+1, \quad k>p.
\end{equation}
Notice also that $\F$ is Cohen-Macaulay if and only if
$Z_k=\emptyset$ for all $k>p$.

\begin{proof}[Proof of Theorem~\ref{oxy}]
If $\F$ is pure,  using \eqref{plast},
it follows as in (the proof of) Lemma~5.2 in \cite{AW2},
that $\ann R_p=\ann R$, and thus  $\ann R_p=\I$. Conversely,
assume that $\ann R_p=\I$.
It follows from Proposition~5.3 in \cite{AW2} that
$\ann R_p$ must be an intersection of primary modules
of codimension $p$, and hence $\I=\ann R_p$ is pure.
\end{proof}

\begin{proof}[Proof of Theorem~\ref{orm}]
Since $R\phi=0$ for $\phi\in\I$, it follows from 
\eqref{opera} that the first mapping in \eqref{kmap2} is
well-defined, and in view of \eqref{astrid} 
the composed mapping coincides with the natural isomorphism.
It  follows that the first mapping must be injective.
(This is also easily seen by a direct argument
that avoids Malgrange's theorem: Assume that $\xi R_k=\dbar\gamma$
for some $\gamma$ in $\Homs(\F,\PM^{0,k-1})$. Then 
$\xi R_k=\nabla^*\gamma$, so that 
$
\nabla^*(U^*\xi-\gamma)=\xi
$
and hence $\xi=0$ in $\H^k(\Ok(E_\bullet^*))$ in view of the first isomorphism
in \eqref{isos}.) 

\smallskip

If $\tilde R$ denotes the current associated with another metric, then as before 
there is a current $M$ with support on $Z$ such that
$$
\nabla_{\End} M=R-\tilde R.
$$
In fact, if 
$$
u=\sigma/\nabla_{\End}\sigma=\sigma+(\dbar\sigma)\sigma +(\dbar\sigma)^2\sigma+\cdots
$$
and  $\tilde u$ is  the analogous form corresponding to the new  metric, then, cf., \cite{AW1}, 
we can take 
\begin{equation}\label{homotopi}
M=\dbar|F|^{2\lambda}\w u\tilde u|_{\lambda=0}.
\end{equation}
Now,  if $\xi\in\Ok(E_k^*)$ and $f_{k+1}^*\xi=0$, then
$$
\xi(R_k-\tilde R_k)=
\xi(\nabla_{\End} M)^0=
\xi(f_{k+1} M^0_{k+1}-\dbar M^0_k)=-\pm \dbar(\xi M^0_k).
$$
Thus we must  show that $\xi M^0_k\phi=0$ for $\phi\in\I$.
If $\phi\in\I$, then $\phi=f_1\psi=\nabla\psi$ for some 
$\psi\in\Ok(E_1)$. Thus
\begin{multline*}
\xi M^0_k\phi=\xi M(\nabla\psi)=\xi(\nabla_{\End}M)\psi-\xi\nabla(M\psi)= \\
\xi(R^1_k-\tilde R^1_k)\psi-
\xi f_{k+1}M^1_{k+1}\psi+\xi\dbar M_k^1\psi=\xi\dbar M^1_k\psi
\end{multline*}
since $R^1=\tilde R^1=0$,
so it is enough to check that  $M^1_k=0$, which we prove by induction
over $k$:
First notice that $M^1_k$ must vanish for $k\le p+1$ since it has bidegree
$(0,k-2)$ and  has support on $Z$ that has  codimension $p$.
Now suppose that we have proved that $M^1_k=0$.
Outside $Z_{k+1}$  the mapping $\sigma_{k+1}$ is smooth so we have, cf., \eqref{homotopi}
and the definition of $R$, 
$$
M^1_{k+1}=\sigma_{k+1}\tilde R^1_{k+1}+(\dbar\sigma_{k+1})M^1_k=(\dbar\sigma_{k+1})M^1_k
=0$$
there since $\tilde R^1=0$. Thus $M^1_{k+1}$ has support on
$Z_{k+1}$ and for degree reasons  it must vanish identically.
\end{proof}

Outside the set $Z_k$, there is a resolution \eqref{acomplex}
of $\F$  with $N<k$, and it follows that then
$\Exts^k(\F,\Ok)\simeq\H^k(\Ok(E_\bullet^*))$ vanishes there; i.e.,
$\Exts^k(\F,\Ok)$ has its support on $Z_k$. 
On the other hand, if $\Exts^\ell(\F,\Ok)=0$ for all $\ell\ge k$, then
$$
\Ok(E^*_{k-1})\stackrel{f^*_k}{\longrightarrow}
\Ok(E^*_k)
\stackrel{f^*_{k+1}}{\longrightarrow}\ldots
\stackrel{f^*_N}{\longrightarrow}\Ok(E_N^*)\to 0
$$
is exact, and hence all the mappings must have constant rank, so we must
be outside $Z_k$. It follows that
\begin{equation}\label{puma}
\supp\,  \Exts^k(\F,\Ok)\subset Z_k\subset \bigcup_{\ell\ge k}
\supp\, \Exts^\ell(\F,\Ok).
\end{equation}
If $\F$ has pure codimension $p$, then it follows from
\eqref{puma} and Eisenbud's theorem mentioned above that
the support of $ \Exts^k(\F,\Ok)$ has at least codimension
$k+1$; a fact that was already 
established  by Roos in \cite{Roos}.

On the other hand, if we have not pure codimension somewhere, then for some $k>p$,
$\codim Z_k$ has codimension $k$.
It follows from \eqref{puma} that then the support $V$ of
$\Exts^k(\F,\Ok)$ has codimension
$k$ here (since $\supp \Exts^\ell(\F,\Ok)$ for $\ell>k$ have higher codimension).
By the coherence there is $\xi\in\Ok(E_k^*)$ whose cohomology class 
is (generically) nonvanishing on $V$. By Theorem~\ref{orm} then the current
$\xi R_k$ represents the corresponding  nonvanishing
class in $\Homs(\F,\Cu^{0,\bullet})$.




\def\listing#1#2#3{{\sc #1}:\ {\it #2},\ #3.}

\end{document}